\documentclass[12pt]{article}
\usepackage[utf8]{inputenc}\usepackage[english]{babel}
\usepackage[T1]{fontenc}\usepackage{amsmath}
\usepackage{amssymb}\usepackage{amsthm}
\usepackage{amsfonts}\usepackage{mathrsfs}
\usepackage{geometry}\usepackage{enumerate}

\newtheorem{theorem}{Theorem}
\newtheorem{definition}{Definition}
\newtheorem{example}{Example}
\newtheorem{proposition}{Proposition}
\newtheorem{lemma}{Lemma}
\newtheorem{corollary}{Corollary}
\makeatletter\renewcommand{\subsection}{\@startsection{subsection}{1}
{0pt}{3.25ex plus 1ex minus.2ex}{-1em}{\normalfont\normalsize\bf}}\makeatother

\begin{document}

\title{On collectively L-weakly compact sets of operators}
\author{Eduard Emelyanov$^{1}$\\ 
\small $1.$ Sobolev Institute of Mathematics, Novosibirsk, 630090, Russia}
\maketitle

\begin{abstract}
A set of bounded linear operators from a Banach space to a 
Banach lattice is collectively L-weakly compact whenever union of
images of the unit ball is L-weakly compact. We extend the Meyer-Nieberg 
duality theorem to collectively L-weakly compact sets of operators, 
study relations between these sets and collectively almost limited sets, and
discuss the domination problem for collectively compact sets
and collectively L-weakly compact sets.
\end{abstract}



\section{Introduction}

\hspace{5mm}
The theory of L-weakly compact sets and operators was developed by P. Meyer-Nieberg
\cite{Mey1974}. Several versions of this theory were studied recently
by many authors (see \cite{AEG2024,EG2023} and references therein).
We extend the Meyer-Nieberg duality theorem \cite[Satz 3]{Mey1974}
and Theorem 2.6 of \cite{CCJ2014} to collectively L-weakly compact sets of operators,
and shortly discuss the domination problem for collectively compact
and collectively L-weakly compact sets.

Throughout this paper vector spaces are real; operators are linear; 
letters $X$ and $Y$ stand for Banach spaces; 
$E$ and $F$ for Banach lattices; $B_X$ for the closed unit ball of $X$; 
$I_X$ for the identity operator on $X$;
$\text{\rm L}(X,Y)$ (resp., $\text{\rm K}(X,Y)$, $\text{\rm W}(X,Y)$, 
$\text{\rm Lim}(X,Y)$, $\text{\rm aLim}(X,Y)$) 
for the space of bounded (resp., compact, weakly compact, limited, almost limited) 
operators from $X$ to $Y$; 
$$
   \text{\rm absco}(A)=\left\{\sum\limits_{k=1}^{m}\alpha_ka_k:
   \sum\limits_{k=1}^{m}|\alpha_k|\le 1 \ \ \text{\rm and}\ \ a_1,...,a_m\in A\right\}
$$ 
for the absolute convex hull of $A\subseteq X$;
and $\text{\rm sol}(A)=\bigcup\limits_{a\in A}[-|a|,|a|]$ 
for the solid hull of $A\subseteq E$.
We denote the uniform convergence of a sequence of functions  
$(g_n)$ to $g$ on a set $A$ by $g_n\rightrightarrows g(A)$.

A bounded subset $A$ of $Y$ is limited whenever $f_n\rightrightarrows 0(A)$ 
for every \text{\rm w}$^\ast$-null sequence $(f_n)$ in $Y'$. 
J. Bourgain and J. Diestel initiated the study of operators carrying bounded 
sets to limited ones in~\cite{BD1984}, such operators are called limited.
A decade ago J. X. Chen, Z. L. Chen, and G. X. Ji
introduced an almost limited set in a Banach lattice \cite[Definition 2.3]{CCJ2014}, 
which is a disjoint version of the limited set, namely
a bounded $A\subseteq F$ is almost limited whenever $f_n\rightrightarrows 0(A)$ 
for every disjoint \text{\rm w}$^\ast$-null sequence $(f_n)$ in $F'$.

The central notion of the present paper is due to P. Meyer-Nieberg \cite[pp.146-147]{Mey1974}: 
a bounded $A\subseteq F$ is called L-weakly compact 
(briefly, is an \text{\rm Lwc}-set) whenever each disjoint 
sequence in $\text{\rm sol}(A)$ is norm-null;
an operator $T\in\text{\rm L}(X,F)$ is said to 
be L-weakly compact (briefly, $T\in\text{\rm Lwc}(X,F)$) whenever $T(B_X)$ is L-weakly compact.

Following P. M. Anselone and T. W. Palmer \cite{AP1968}, 
we call a subset ${\cal T}$ of $\text{\rm L}(X,Y)$ 
collectively (weakly) compact (collectively (almost) limited, collectively \text{\rm Lwc}) 
if the set ${\cal T}B_X=\bigcup\limits_{T\in{\cal T}}T(B_X)$ 
is relatively (weakly) compact (resp., (almost) limited, \text{\rm Lwc}). 
For convenience, we use bold letters to denote classes of collectively ${\cal P}$-sets of operators
acting from $X$ to $Y$, e.g., $\text{\bf K}(X,Y)$, $\text{\bf W}(X,Y)$, 
$\text{\bf Lim}(X,Y)$, $\text{\bf aLim}(X,F)$. Note that the collectively bounded sets 
of operators of $\text{\bf L}(X,Y)$ are exactly the uniformly bounded sets. 

More generally, collectively ${\cal C-P}$ sets of
operators are defined for a property ${\cal C}$ of subsets of $X$ and 
a property ${\cal P}$ of subsets of $Y$ by saying that a subset ${\cal T}$ of 
$\text{\bf L}(X,Y)$ is collectively ${\cal C}-{\cal P}$ whenever
${\cal T}(C)=\bigcup\limits_{T\in{\cal T}}T(C)$ has the property
${\cal P}$ for each set $C$ with the property ${\cal C}$. 
For example: collectively Gelfand -- Phillips (resp., Bourgain -- Diestel, 
collectively limitedly L-weakly compact) sets 
$\text{\bf GP}(X,Y)$ (resp., $\text{\bf BD}(X,Y)$, $\text{\bf l-Lwc}(X,F)$) correspond 
to the property ${\cal C}$ to be limited and the property 
${\cal P}$ to be relatively compact (resp., relatively weakly compact, \text{\rm Lwc});
collectively order bounded sets $\text{\bf L}_{\text{\bf ob}}(E,F)$
correspond to the properties ${\cal C}$ and ${\cal P}$ to be order bounded; 
etc., (cf. \cite{EG2023,AEG2024}). Clearly, each 
collectively ${\cal C}-{\cal P}$ set consists of ${\cal C}-{\cal P}$ operators,
for example, $T\in{\cal T}\in\text{\bf l-Lwc}(X,F)\Longrightarrow T\in l-\text{\rm Lwc}(X,F)$.

\medskip
For further unexplained notation and terminology, we refer to 
\cite{AB2006,BD1984,CCJ2014,Kus2000,Mey1991}.

\section{Main results}

\hspace{5mm}
We need several technical lemmas, some of them might be already known.
Since we did not find appropriate references for them, we include their proofs.

\begin{lemma}\label{CL-prop}
{\em
The strong closure of the absolute convex hull of a collectively limited (collectively almost limited,
collectively weakly compact) set is collectively limited (resp., collectively almost limited, 
collectively weakly compact).
}
\end{lemma}

\begin{proof}
Let ${\cal T}\in\text{\bf Lim}(X,Y)$.
Since $\overline{{\cal T}}^sB_X\subseteq\overline{{\cal T}B_X}$,
the strong closure $\overline{{\cal T}}^s$ of ${\cal T}$
is collectively limited by \cite[Proposition]{BD1984}.
The collectively almost limited case is similar, and the proof for collectively 
weakly compact sets is a straightforward modification of the proof of
\cite[Proposition 2.1]{AP1968}.
\end{proof}

\noindent
The next lemma is a minor extension of \cite[Corollary 3.6.4~ii)]{Mey1991}.

\begin{lemma}\label{abs convex hull of LWS-sets}
{\em
The absolute convex hull of every \text{\rm Lwc} set is an \text{\rm Lwc} set.
}
\end{lemma}

\begin{proof}
Let $A$ be an \text{\rm Lwc} set in $E$. 
By \cite[Theorem 5.63]{AB2006}, for proving that $\text{\rm absco}(A)$ 
is an \text{\rm Lwc} set it suffices to show that each disjoint sequence 
in $B_{E'}$ is uniformly null on $\text{\rm absco}(A)$.
So, let $(f_n)$ be disjoint in $B_{E'}$, and let $\varepsilon>0$.
Applying \cite[Theorem 5.63]{AB2006}, we have an $n_\varepsilon$ such
that $\sup\limits_{a\in A}|f_n(a)|\le\varepsilon$ for all $n\ge n_\varepsilon$.
Take an arbitrary $x\in\text{\rm absco}(A)$, say
$x=\sum\limits_{k=1}^{m}\alpha_ka_k$, where 
$\sum\limits_{k=1}^{m}|\alpha_k|\le 1$ and $a_1,...,a_m\in A$. Then
$$
   |f_n(a)|=\Big|\sum\limits_{k=1}^{m}\alpha_kf_n(a_k)\Big|\le
   \sum\limits_{k=1}^{m}|\alpha_k||f_n(a_k)|\le\varepsilon \ \ \ \ \ (n\ge n_\varepsilon),
$$
and hence $f_n\rightrightarrows 0(\text{\rm absco}(A))$ as desired.
\end{proof}

\begin{lemma}\label{closure of LWS-sets}
{\em
The norm closure of every \text{\rm Lwc} set is likewise an \text{\rm Lwc} set.
}
\end{lemma}

\begin{proof}
Let $A$ be an \text{\rm Lwc} set in $E$, and let $(\bar{x}_n)$ be a disjoint sequence in 
the solid hull $\text{\rm sol}(\bar{A})$ of the norm closure $\bar{A}$ of $A$. 
Suppose, in contrary, $(\bar{x}_n)$ is not norm null. WLOG, $\|\bar{x}_n\|\ge\varepsilon$
for some $\varepsilon>0$ and all $n$.
For every $n$, pick $\bar{a}_n\in\bar{A}$ such that $|\bar{x}_n|\in[0,|\bar{a}_n|]$ and 
choose $a_n\in A$ satisfying $\|\bar{a}_n-a_n\|\le 2^{-n}$. Then, for all $n$,
$$
   \||\bar{x}_n|\wedge|\bar{a}_n|\|-\||\bar{x}_n|\wedge|a_n|\|\le
   \||\bar{x}_n|\wedge|\bar{a}_n|-|\bar{x}_n|\wedge|a_n|\|\le
   \||\bar{a}_n|-|a_n|\|\le\|\bar{a}_n-a_n\|\le 2^{-n}.
$$
Since $(|\bar{x}_n|\wedge|a_n|)$ is disjoint in 
$\text{\rm sol}(A)$ then $\||\bar{x}_n|\wedge|a_n|\|\to 0$
which is absurd as
$$
   \||\bar{x}_n|\wedge|a_n|\|+2^{-n}\ge\||\bar{x}_n|\wedge|\bar{a}_n|\|=
   \||\bar{x}_n|\|=\|\bar{x}_n\|\ge\varepsilon
$$
for all $n$. The obtained contradiction completes the proof.
\end{proof}

\medskip
Now, we are in the position to prove an \text{\rm Lwc}-version
of \cite[Proposition 2.1]{AP1968}.

\begin{proposition}\label{CLWC-prop}
{\em
The strong closure of absolute convex hull of
a collectively \text{\rm Lwc} set is likewise collectively \text{\rm Lwc}.
}
\end{proposition}

\begin{proof}
Let ${\cal T}\in\text{\bf Lwc}(X,F)$. 
Since $\overline{\text{\rm absco}({\cal T})}^sB_X\subseteq\overline{\text{\rm absco}({\cal T})B_X}$,
$\overline{\text{\rm absco}({\cal T})}^s$ is collectively \text{\rm Lwc}
by Lemmas \ref{abs convex hull of LWS-sets} and \ref{closure of LWS-sets}.
\end{proof}

\medskip
The following result provides collectively (almost) limited,  weakly compact, and 
\text{\rm Lwc} versions of \cite[Propositions 2.2]{AP1968}.
We suppose that $\lambda$ is the Lebesgue measure on the $\sigma$-algebra of measurable subsets of 
$[a,b]\subset\mathbb{R}$; $K_\xi:[a,b]\to\text{\rm L}(X,Y)$ is a $\lambda$-measurable function
(cf., \cite[Definition 3.9.12]{Mey1991}); and 
if $K_\xi$ is Bochner integrable then $\int_a^bK_\xi(t)dt$ is the Bochner integral
(cf., \cite[Def.3.9.15]{Mey1991}).

\begin{proposition}\label{CLWC-int}
{\em
With the foregoing notation, let $K_\xi:[a,b]\to\text{\rm L}(X,Y)$ be Bochner integrable for $\xi\in\Xi$.
The following statements hold:
\begin{enumerate}[(i)]
\item 
If $\{K_\xi(t)\}_{\xi\in\Xi, a\le t\le b}$ is collectively (almost) limited (collectively weakly compact)
in $\text{\rm L}(X,Y)$ then 
$\Big\{\int_a^bK_\xi(t)dt\Big\}_{\xi\in\Xi}$ is likewise collectively (almost) limited 
(resp., collectively weakly compact).
\item 
If $\{K_\xi(t)\}_{\xi\in\Xi, a\le t\le b}\in\text{\bf Lwc}(X,F)$ then $\Big\{\int_a^bK_\xi(t)dt\Big\}_{\xi\in\Xi}\in\text{\bf Lwc}(X,F)$.
\end{enumerate}
}
\end{proposition}

\begin{proof}
(i) follows from Lemma \ref{CL-prop}, whereas (ii) is a consequence 
of Proposition \ref{CLWC-prop}.
\end{proof}

\medskip
Each element of a collectively (almost) limited (collectively weakly compact, collectively \text{\rm Lwc})
set is an (almost) limited (resp., weakly compact, \text{\rm Lwc}) operator.
The converse fails as can be seen by considering the set of all normalized rank one operators in $\ell^1$.
However, similarly to \cite[Theorem 2.5]{AP1968} we have:

\begin{theorem}\label{CLWC-compact}
{\em
Every totally bounded set of (almost) limited (weakly compact, \text{\rm Lwc})
operators is collectively (almost) limited 
(resp., collectively weakly compact, collectively \text{\rm Lwc}).
}
\end{theorem}

\begin{proof}
We include the proof only for \text{\rm Lwc} operators. Remaining cases are left for the reader.
So, let ${\cal T}\subseteq\text{\rm Lwc}(X,F)$ be totally bounded. 
By the equivalence $i)\Longleftrightarrow iii)$ of
\cite[Proposition 3.6.2]{Mey1991}, we obtain that ${\cal T}B_X$ is an \text{\rm Lwc} set.
Now, the applying of Proposition~\ref{CLWC-prop} completes the proof.
\end{proof} 
\noindent
The converse of Theorem \ref{CLWC-compact} is false. 
The following two examples are motivated by \cite[Example 2.6]{AP1968}.

\begin{example}\label{dual-in-example}
{\em 
Let $T_n\in\text{\rm L}(\ell^1)$ be defined by $T_nx=x_ne_1$, where 
$x=\sum\limits_{k=1}^\infty x_k e_k\in\ell^1$ and $e_k$ is the k-th standard unit vector in $\ell^1$.
Clearly, the set $\{T_n\}_{n=1}^\infty$ is collectively \text{\rm Lwc} (compact, weakly compact, and limited). But $\{T_n\}_{n=1}^\infty$ is not totally bounded,
for $\|T_n-T_m\|=2$ with $n\ne m$. 

The adjoint set
$\{T'_n\}_{n=1}^\infty\subset\text{\rm L}(\ell^\infty)$ is collectively limited,
yet not collectively compact. Indeed, 
$$
   (T'_nf)x=f(T_nx)=f(x_ne_1)=x_nf(e_1)=x_nf_1 \ \ \ \ \ (f\in\ell^\infty, x\in\ell^1),
$$
and hence $T'_nf=f_1e_n$. Thus,
$
  \bigcup\limits_{n\in\mathbb{N}}T'_n\bigl(B_{\ell^\infty}\bigl)=
  \bigcup\limits_{n\in\mathbb{N}}[-e_n,e_n]\subset c_0
$. 
Therefore, $\{T'_n\}_{n=1}^\infty$ is collectively limited due to
Phillip’s lemma (cf. \cite[Theorem~4.67]{AB2006}).
Since $e_n\in T'_nB_{\ell^\infty}$ for every $n$,
the set $\{T'_n\}_{n=1}^\infty$ is not collectively compact.

Note that, since the set $\{T_n\}_{n=1}^\infty$ is collectively \text{\rm Lwc}
then $\{T'_n\}_{n=1}^\infty$ is collectively \text{\rm Mwc} by
Theorem \ref{CLW-MW-duality}. Thus, $\{T'_n\}_{n=1}^\infty$ is collectively
weakly compact by Corollary \ref{CLW-MW-are-CWC}.

Now, let $S_n\in\text{\rm L}(c_0)$ be defined similarly by $S_nx=x_ne_1$.
Again, $\{S_n\}_{n=1}^\infty$ is collectively compact.
However, in this case $\{S'_n\}_{n=1}^\infty\subset\text{\rm L}(\ell^1)$ is not collectively
weakly compact, since 
$
  \{e_n\}_{n=1}^\infty\subset\bigcup\limits_{n\in\mathbb{N}}[-e_n,e_n]=
  \bigcup\limits_{n\in\mathbb{N}}S'_n\bigl(B_{\ell^1}\bigl)
$
and the subset $\{e_n\}_{n=1}^\infty$ of $\ell^1$ is
not relatively weakly compact.}
\end{example}

\begin{example}\label{example}
{\em 
Let $\dim(F^a)=\infty$, where $F^a$ is the order continuous part of $F$. 
Take a disjoint normalized sequence $(u_n)$ in $F^a_+$
and define rank one operators $T_n:\ell^\infty\to F$ by $T_nx=x_nu_n$. 
Every $T_n$ is an \text{\rm Lwc} operator since $T_n(\ell^\infty)\subseteq F^a$. 
However, ${\cal T}=\{T_n\}_{n=1}^{\infty}\notin\text{\bf Lwc}(\ell^\infty,F)$.
Indeed, for the bounded set $A=\{e_n\}_{n=1}^{\infty}\subset\ell^\infty$,
the set ${\cal T}A=\{0\}\cup\{u_n\}_{n=1}^{\infty}$ is not \text{\rm Lwc}.}
\end{example}

\medskip
By the Meyer-Nieberg result \cite[Satz 3]{Mey1974}, 
the L- and M-weakly compact operators are in duality to each other
(cf. \cite[Proposition 3.6.11]{Mey1991}, \cite[Theorem 5.64]{AB2006}).
We need the following collective analogue of the notion of \text{\rm Mwc} operators.

\begin{definition}\label{Mwc-sets}
{\em
A bounded subset ${\cal T}$ of $\text{\rm L}(E,Y)$ is said to be collectively M-weakly compact 
(briefly, collectively \text{\rm Mwc}, or ${\cal T}\in\text{\bf Mwc}(E,Y)$) 
if $\lim\limits_{n\to\infty}\sup\limits_{T\in{\cal T}}\|Tx_n\|=0$
holds for every norm bounded disjoint sequence $(x_n)$ of $E$.
}
\end{definition}

\noindent
The following result is a generalization of the Meyer-Nieberg 
duality theorem to collectively L(M)-weakly compact sets.

\begin{theorem}\label{CLW-MW-duality}
{\em 
The following statements hold:
\begin{enumerate}[(i)]
\item 
${\cal T}\in\text{\bf Lwc}(X,F)\Longleftrightarrow{\cal T}'=\{T':T\in{\cal T}\}\in\text{\bf Mwc}(F',X')$.
\item 
${\cal T}\in\text{\bf Mwc}(E,Y)\Longleftrightarrow{\cal T}'\in\text{\bf Lwc}(Y',E')$.
\end{enumerate}
}
\end{theorem}

\begin{proof}
(i)\
${\cal T}\in\text{\bf Lwc}(X,F)$ iff $\bigcup\limits_{T\in{\cal T}}T(B_X)$ is bounded and
every disjoint sequence of $\text{\rm sol}\left(\bigcup\limits_{T\in{\cal T}}T(B_X)\right)$
is uniformly null on $B_{F'}$ iff ${\cal T}$ is bounded and
$f_n\rightrightarrows 0\left(\bigcup\limits_{T\in{\cal T}}T(B_X)\right)$
for every disjoint $(f_n)$ in $B_{F'}$
iff ${\cal T}'$ is bounded and $\lim\limits_{n\to\infty}\sup\limits_{T\in{\cal T}}\|T'f_n\|=0$
holds for every disjoint $(f_n)$ of $B_{F'}$ by \cite[Theorem.5.63]{AB2006} 
iff ${\cal T}'\in\text{\bf Mwc}(F',X')$.

(ii)\
${\cal T}\in\text{\bf Mwc}(E,Y)$ iff  ${\cal T}$ is bounded and 	
$\lim\limits_{n\to\infty}\sup\limits_{T\in{\cal T}}\|Tx_n\|=0$
for every disjoint sequence $(x_n)$ in $B_E$ iff
$x_n\rightrightarrows 0\left(\bigcup\limits_{T\in{\cal T}}T'(B_{Y'})\right)$ 
for every disjoint $(x_n)$ in $B_E$, and by 
\cite[Theorem 5.63]{AB2006}, iff every disjoint sequence of 
$\text{\rm sol}\left(\bigcup\limits_{T\in{\cal T}}T'(B_{Y'})\right)$
is norm null, iff ${\cal T}'\in\text{\bf Lwc}(Y',E')$.
\end{proof}
\noindent
The following corollary generalizes another Meyer-Nieberg result (cf. \cite[Theorem 5.61]{AB2006})
to collectively L(M)-weakly compact sets.

\begin{corollary}\label{CLW-MW-are-CWC}
{\em 
Every collectively \text{\rm Lwc} (\text{\rm Mwc}) set is 
collectively weakly compact.
}
\end{corollary}

\begin{proof}
It is a consequence of \cite[Proposition 3.6.5]{Mey1991}
for collectively \text{\rm Lwc} sets. Now, for collectively \text{\rm Mwc} sets,
the result follows from Theorem \ref{CLW-MW-duality}.
\end{proof}

\medskip
Next, we present a collective version of \cite[Theorem 2.6]{CCJ2014}.

\begin{theorem}\label{CCJ-collective}
{\em 
For every $X$ and $F$, $\text{\bf Lwc}(X,F)\subseteq\text{\bf aLim}(X,F)$.
Furthermore, the following conditions are equivalent.
\begin{enumerate}[(i)]
\item 
$F$ has order continuous norm.
\item
For every $X$, $\text{\bf aLim}(X,F)\subseteq\text{\bf Lwc}(X,F)$.
\item 
$\text{\bf aLim}(F)\subseteq\text{\bf Lwc}(F)$.
\end{enumerate}
}
\end{theorem}

\begin{proof}
Let ${\cal T}$ be a collectively \text{\rm Lwc} subset of $\text{\rm L}(X,F)$.
By Theorem \ref{CLW-MW-duality}, ${\cal T}'\subseteq\text{\rm L}(F',X')$
is collectively \text{\rm Mwc}. Then, for every disjoint $(f_n)$ in $B_{F'}$,  
$$
   \lim\limits_{n\to\infty}\sup\{|f_n(Tx)|: x\in B_X,T\in{\cal T}\}=
   \lim\limits_{n\to\infty}\sup\limits_{T\in{\cal T}}\|T'f_n\|=
   \lim\limits_{n\to\infty}\sup\limits_{S\in{\cal T}'}\|Sf_n\|=0,
$$
and hence ${\cal T}B_X=\bigcup\limits_{T\in{\cal T}}T(B_X)$ 
is almost limited, i.e. ${\cal T}$ is collectively almost limited.

\medskip
(i)$\Longrightarrow$(ii)\ \
Let ${\cal T}\in\text{\bf aLim}(X,F)$. 
Since the set ${\cal T}B_X=\bigcup\limits_{T\in{\cal T}}T(B_X)$ is almost limited,
it is an \text{\rm Lwc} set by \cite[Theorem 2.6]{CCJ2014}, 
and hence ${\cal T}$ is collectively \text{\rm Lwc}.

\medskip
(ii)$\Longrightarrow$(iii)\ \
It is obvious.

\medskip
(iii)$\Longrightarrow$(i)\ \
Now, assume that every collectively almost limited subset of $\text{\rm L}(F)$
is collectively \text{\rm Lwc}. Then each almost limited operator is \text{\rm Lwc},
and hence the norm in $F$ is order continuous by \cite[Theorem 2.6]{CCJ2014}.
\end{proof}

\medskip
We turn our discussion to the domination problem 
for collectively compact sets.

\begin{definition}\label{collective domination}
{\em
Let ${\cal S}$ and ${\cal T}$ be two subsets of $\text{\rm L}_+(E,F)$.
Then ${\cal S}$ is dominated by ${\cal T}$ if, for each $S\in{\cal S}$
there exists $T\in{\cal T}$ such that $S\le T$.}
\end{definition}

\noindent
The proofs of the following two propositions are straightforward modifications
of proofs of \cite[Theorem 5.11]{AB2006} and \cite[Lemma 5.12]{AB2006}
respectively.

\begin{proposition}\label{|sigma|-totally bounded}
{\em
Let ${\cal S}\subset\text{\rm L}_+(E,F)$ be dominated by $\{T\}$ for some 
$T\in\text{\rm L}_+(E,F)$.
If $T[0,x]$ is $|\sigma|(F,F')$-totally bounded
for each $x\in E_+$, then 
$\bigcup\limits_{S\in{\cal S}}S[0,x]$ is likewise $|\sigma|(F,F')$-totally bounded
for each $x\in E_+$.
}
\end{proposition}

\begin{proposition}\label{collective domination}
{\em
Let ${\cal S}\subset\text{\rm L}_+(E,F)$ be dominated by $\{T\}$ for some 
$T\in\text{\rm L}_+(E,F)$.
If $T(A)$ is 
a norm totally bounded set, then for each $\varepsilon>0$
there exists some $u\in F_+$ such that
$\sup\limits_{S\in{\cal S}}\bigl\|(Sx-u)^+\bigl\|\le\varepsilon$
holds for all $x\in A$.
}
\end{proposition}

\noindent
The author does not know whether or not a single operator $T$
in Proposition \ref{|sigma|-totally bounded} (in Proposition \ref{collective domination}) 
can be replaced by a set ${\cal T}\subset\text{\rm L}_+(E,F)$ dominating ${\cal S}$
such that $\bigcup\limits_{T\in{\cal T}}T[0,x]$ is $|\sigma|(F,F')$-totally bounded
(resp., $\bigcup\limits_{T\in{\cal T}}T(A)$ is norm totally bounded).

\noindent
Replacing \cite[Theorem 5.11]{AB2006} by Proposition \ref{|sigma|-totally bounded}
and \cite[Lemma 5.12]{AB2006} by Proposition \ref{collective domination} in the proof
of \cite[Theorem 5.13]{AB2006} gives a collective extension of the
Aliprantis--Burkinshaw theorem as follows.

\begin{theorem}\label{AB-S3}
{\em
Let ${\cal S}\subset\text{\rm L}_+(E)$ be dominated by $\{T\}$ for some 
$T\in\text{\rm K}_+(E)$, then the set $\{S^3:S\in{\cal S}\}$
is collectively compact.
}
\end{theorem}

\noindent
A similar modification of the proof of \cite[Theorem 5.20]{AB2006} gives 
the following collective extension of the Dodds--Fremlin theorem.

\begin{theorem}\label{DF-Domi}
{\em
Let the norms in $E'$ and in $F$ be order continuous.
If ${\cal S}\subset\text{\rm L}_+(E,F)$ is dominated by $\{T\}$ for some 
$T\in\text{\rm K}_+(E,F)$, then ${\cal S}\in\text{\bf K}(E,F)$.
}
\end{theorem}

\noindent
The author does not know whether or not the operator $T\in\text{\rm K}_+(E,F)$
in Theorems \ref{AB-S3} and \ref{DF-Domi} can be replaced by 
a collectively compact set ${\cal T}\subset\text{\rm L}_+(E,F)$ dominating ${\cal S}$.

\medskip
Finally, we mention briefly the domination problem for collectively \text{\rm Lwc} sets.
Whereas collective extensions of the Aliprantis--Burkinshaw theorem
\cite[Theorem~5.11]{AB2006} and of the Dodds--Fremlin theorem
\cite[Theorem~5.20]{AB2006} are not sufficiently clear yet, 
the domination problem for collectively \text{\rm Lwc} sets is much easier
and has the complete solution, as follows.

\begin{theorem}\label{LW-Domi}
{\em Let
${\cal S}\subset\text{\rm L}_+(E,F)$ be dominated by ${\cal T}\subset\text{\rm L}_+(E,F)$.
\begin{enumerate}[(i)]
\item 
If ${\cal T}$ is collectively \text{\rm Lwc}, then ${\cal S}$ is collectively \text{\rm Lwc}.
\item 
If ${\cal T}$ is collectively \text{\rm Mwc}, then ${\cal S}$ is collectively \text{\rm Mwc}.
\end{enumerate}
}
\end{theorem}

\begin{proof}
(i)\ \ 
Let ${\cal T}$ be a collectively \text{\rm Lwc} set.
Since ${\cal S}$ is dominated by ${\cal T}$, then
$Sx\in\bigcup\limits_{T\in{\cal T}}T(B^+_E)$ for each $x\in B^+_E$.
As $\bigcup\limits_{S\in{\cal S}}S(B^+_E)\subseteq
\text{\rm sol}\left(\bigcup\limits_{T\in{\cal T}}T(B_E)\right)$,
and $\bigcup\limits_{T\in{\cal T}}T(B_E)$ is an \text{\rm Lwc} subset of $F$,
therefore, $\bigcup\limits_{S\in{\cal S}}S(B^+_E)$, and hence
$\bigcup\limits_{S\in{\cal S}}S(B_E)$ is also \text{\rm Lwc}
by \cite[Corollary 3.6.4]{Mey1991}.

\medskip
(ii) Apply Theorem \ref{CLW-MW-duality}.
\end{proof}


\bibliographystyle{plain}
\end{document}